\documentclass[11pt]{article}
\usepackage{graphicx}
\usepackage{amssymb,amsmath,amsthm,stmaryrd}
\usepackage[latin1]{inputenc}
\usepackage[english]{babel}
\usepackage{color}
\usepackage{vmargin, enumerate}
\def\Z{\mathbb{Z}}

\def\le{\leqslant}
\def\ge{\geqslant}

\def\eps{\varepsilon}

\def\->{\rightarrow}
\def\<{\langle}
\def\>{\rangle}

\newtheorem{theorem}{Theorem}[section]
\newtheorem{lemma}[theorem]{Lemma}
\newtheorem{proposition}[theorem]{Proposition}

\newtheorem{obs}[theorem]{Observation}

\newtheorem{conjecture}[theorem]{Conjecture}

\theoremstyle{remark}

\theoremstyle{remark}
\theoremstyle{remark}

\newcommand{\ssq}{\subseteq}
\newcommand{\subgp}[1]{\langle{#1}\rangle}

\begin{document}

\title{Subset sums in abelian groups}
\date{ }
\author{
Eric Balandraud
\thanks{IMJ, \'{E}quipe Combinatoire et Optimisation, Universit\'{e} Pierre et Marie Curie~(Paris $6$), $4$ place Jussieu, 75005 Paris, France, email: \texttt{balandraud@math.jussieu.fr}.}
\and
Benjamin Girard
\thanks{IMJ, \'{E}quipe Combinatoire et Optimisation, Universit\'{e} Pierre et Marie Curie~(Paris $6$), $4$ place Jussieu, 75005 Paris, France, email: \texttt{bgirard@math.jussieu.fr}.}
\and
Simon Griffiths
\thanks{IMPA, Estrada Dona Castorina $110$, Rio de Janeiro, Brasil, $22460$-$320$, email: \texttt{sgriff@impa.br}. Research supported by CNPq Proc. 500016/2010-2.}
\and
Yahya ould Hamidoune
}

\maketitle
\vspace{-0.6cm}

\begin{center}\emph{To Yahya ould Hamidoune, an inspiration and a dear friend.}\end{center}\vspace{0.8cm}

\begin{abstract} 
Denoting by $\Sigma(S)$ the set of subset sums of a subset $S$ of a finite abelian group $G$, we prove that 
\[
|\Sigma(S)|\ge \frac{|S|(|S|+2)}{4}-1
\]
whenever $S$ is symmetric, $|G|$ is odd and $\Sigma(S)$ is aperiodic.  
Up to an additive constant of $2$ this result is best possible, and we obtain the stronger (exact best possible) bound in almost all cases.  
We prove similar results in the case $|G|$ is even.  
Our proof requires us to extend a theorem of Olson on the number of subset sums of anti-symmetric subsets $S$ from the case of $\mathbb{Z}_p$ to the case of a general finite abelian group.
To do so, we adapt Olson's method using a generalisation of Vosper's Theorem proved by Hamidoune and Plagne. 
\end{abstract}

\section{Introduction}
\label{Section Introduction}

The study of the set of subset sums
\begin{equation}\label{def:Sigma}
\Sigma(S)\, =\, \left\{\sum_{x\in X}x:\, X\subseteq S\right\}\, 
\end{equation}
of a subset $S$ of a finite abelian group $G$ is well established within the field of Additive Number Theory and was a recurring theme in the research of Yahya ould Hamidoune through the years.  His contributions here, and on the related problem of the restricted sumset, have greatly increased our understanding.

The study of subset sums may be traced back to the 1964 paper of Erd\H os and Heilbronn~\cite{EH64}.  They consider the question of determining the minimum $\ell\in \mathbb{N}$ such that every subset $S\subseteq \mathbb{Z}_p\setminus \{0\}$ ($p$ prime) with $|S|\ge\ell$ covers $\mathbb{Z}_p$ with its subset sums, i.e., satisfies $\Sigma(S)=\mathbb{Z}_p$.  They proved that $\Sigma(S)=\mathbb{Z}_p$ provided $|S|\ge 3\sqrt{6}\sqrt{p}$.

The same question may be considered in an arbitrary finite abelian group.  In fact, in this case the \emph{critical number} of a finite abelian group $G$,
\[
\mathsf{cr}(G)=\min\{\ell: \Sigma^*(S)=G \,\, \text{for all}\, S\subseteq G\setminus \{0\},\, |S|\ge\ell\}\, ,
\]
is defined in terms of $\Sigma^*(S)$, the set of all non-empty subset sums of $S$, but this difference is not of any great importance to the present discussion.
Improving on the result of Erd\H os and Heilbronn~\cite{EH64}, Olson~\cite{Olson68} proved that $\mathsf{cr}(\mathbb{Z}_p)\le 2\sqrt{p}$.  The precise result that $\mathsf{cr}(\mathbb{Z}_p)=\lfloor 2(\sqrt{p-2})\rfloor$ for all primes $p\ge 3$ follows from Theorem 4.2 and Example 4.2 of Dias da Silva and Hamidoune~\cite{DH94} (using the observation that $4p-7$ is not a square for any prime $p\ge 3$).  The critical number is now known precisely for every finite abelian group, see the article of Freeze, Gao and Geroldinger~\cite{FGG09} and the references contained therein.

A closely related problem to the determination of $\mathsf{cr}(G)$ is the problem of proving bounds on $|\Sigma(S)|$.  Indeed, Erd\H os and Heilbronn~\cite{EH64} proved their bound on $\mathsf{cr}(\mathbb{Z}_p)$ by proving a quadratic lower bound on $|\Sigma(S)|$ for subsets $S\subseteq \mathbb{Z}_p$ and Olson~\cite{Olson68} improved on their result by proving the following lower bound on $|\Sigma(S)|$.  We remark that the bound is best possible in almost all cases, the exceptional case being the case when $|\Sigma(S)|$ is almost as large as $|G|/2$ in which case  $\xi(S)=0$.  For a subset $S$ of a finite abelian group $G$, we denote by $\subgp{S}$ the subgroup generated by $S$, and the parameter $\xi(S)$ is defined to be identically $1$ if $|S|$ is even and as follows in the case $|S|$ is odd:
\[
\xi(S)\, = 
\begin{cases}
1 &\qquad \text{ if } 2|S|^2 +3|S| \le 2|\subgp{S}|+5 \\ 
0 &\qquad \text{ if } 2|S|^2 +3|S| > 2|\subgp{S}|+5 \, .
\end{cases}
\]

\begin{theorem}[Olson]\label{thm:Olson}
Let $G=\mathbb{Z}_{p}$ where $p$ is prime, and let $S$ be a subset of $G$ such that $S\cap (-S)=\emptyset$. Then one of the following holds.
\begin{itemize}
\item[$(i)$]
\[
|\Sigma(S)|\ge \frac{|S|(|S|+1)}{2}+\xi(S) \, .
\]
\item[$(ii)$]
\[
|\Sigma(S)|>\frac{p}{2}\, .
\]
\end{itemize}
\end{theorem}\vspace{0.3cm}

In the case of a general finite abelian group, non-trivial subgroups present an obstacle to extending Olson's Theorem.  
For this reason we consider the following to be the natural extension of Olson's Theorem to the case of a general finite abelian group.  

\begin{theorem}\label{thm:main}
Let $G$ be a finite abelian group, and let $S\subseteq G$ be such that $S\cap (-S)=\emptyset$ and $|S|\ge 2$. Then one of the following holds.
\begin{itemize}
\item[($i$)] 
\[
|\Sigma(S)|\ge \frac{|S|(|S|-1)}{2}+3\, .
\]
\item[$(ii)$] There is a non-empty subset $S'\subseteq S$ for which
\[
|\Sigma(S')|>\frac{|\subgp{S'}|}{2}\, .
\]
\end{itemize}
Furthermore, if $|G|$ is odd then property ($i$) may be replaced by
\begin{itemize}
\item[($i'$)]
\[
|\Sigma(S)|\ge \frac{|S|(|S|+1)}{2}+\xi(S) \, .
\]
\end{itemize}
\end{theorem}\vspace{0.3cm}

We now describe a consequence of Theorem~\ref{thm:main} (see Theorem~\ref{thm:sym} below) that was in fact our main motivation for proving it.

The fact that $|\Sigma(S)|$ exhibits quadratic growth as a function of $|S|$ was established by Erd\H os and Heilbronn for subsets $S\subseteq \mathbb{Z}_p$.  The analogous result for general finite abelian groups was established by DeVos, Goddyn, Mohar and \v{S}\'{a}mal \cite{DGM07}.  We say that a subset $X$ of a finite abelian group $G$ is \emph{aperiodic} if the equality $X+g=X$ is satisfied only for $g=0$.

\begin{theorem}[DeVos, Goddyn, Mohar and \v{S}\'{a}mal] \label{thm:dvgms} Let $G$ be a finite abelian group, and $S\subseteq G\setminus \{0\}$ a subset for which $\Sigma(S)$ is aperiodic.  Then $|\Sigma(S)|\ge |S|^2/64$.\end{theorem}

It is believed that $\frac{1}{64}$ may be replaced by $\frac{1}{4}$.  The natural extremal example that shows that $\frac{1}{4}$ would be best possible is the subset $S=\{\pm 1,\dots ,\pm s\}\subseteq \mathbb{Z}_n$, where $n> s(s+1)+1$.  This set $S$ has $|S|=2s$ and $|\Sigma(S)|=s(s+1)+1=s^2+s+1$.  We note that the set $S$ is symmetric (i.e., $S=-S$) and remark that we believe in general that such extremal examples should be symmetric (or very close to symmetric).  This belief is supported by the fact~\cite{Gr09}, that we may replace the fraction $\frac{1}{64}$ of Theorem~\ref{thm:dvgms} by $\frac{1}{4}-o(1)$ in general and by $\frac{1}{2}-o(1)$ in the case that $S\cap(-S)=\emptyset$.  Indeed, by adapting the approach of~\cite{Gr09} slightly one obtains that $\frac{1}{64}$ may be replaced by $\frac{1}{4}$ provided that $S$ is large and far from being symmetric.

\begin{theorem}\label{thm:nonsym}  For all $\epsilon >0$ there exists a constant $n_0=n_0(\epsilon)$ such that the following holds.  Let $G$ be a finite abelian group, and $S\subseteq G\setminus \{0\}$ a subset with $|S\triangle (-S)|\ge \eps |S|$, $|S|\ge n_0$ and for which $\Sigma(S)$ is aperiodic.  Then $|\Sigma(S)|\ge (\frac{1}{4}+\epsilon^2)|S|^2$.\end{theorem}

We hope it is now clear to the reader that symmetric sets $S\subseteq G$ are of particular interest.  We may deduce from Theorem~\ref{thm:main} the following bounds on the number of subset sums of symmetric sets.  For a symmetric set $S$ we write $\xi'(S)$ for $\xi(S')$ where $S'$ is any subset of $S$ with $|S'|=|S|/2$ and $S=S'\cup(-S')$.  Equivalently
\[
\xi'(S)\, = 
\begin{cases}
1 & \text{ if } \frac{1}{2}|S|^2 + \frac{3}{2}|S| \le 2|\subgp{S}|+5 \\ 
0 & \text{ if } \frac{1}{2}|S|^2 + \frac{3}{2}|S| > 2|\subgp{S}|+5 \, .
\end{cases}
\]

\begin{theorem}\label{thm:sym}  Let $G$ be a finite abelian group, and $S\subseteq G\setminus \{0\}$ a symmetric subset with $|S|\ge 4$ for which $\Sigma(S)$ is aperiodic.  Then 
\[
|\Sigma(S)|\ge \frac{|S|(|S|-2)}{4}+5 \, .
\]
Furthermore, if $|G|$ is odd then
\[
|\Sigma(S)|\ge \frac{|S|(|S|+2)}{4}+2\xi'(S)-1 \, .
\]
\end{theorem}

By considering the example $S=\{\pm 1,\dots ,\pm s\}\subseteq \mathbb{Z}_n$ given above we observe that the latter bound is best possible (except in the exceptional case that $\xi'(S)=0$).  We conjecture that this bound should hold even if the conditions that $S$ is symmetric and $|G|$ is odd are dropped.

\begin{conjecture} Let $G$ be a finite abelian group, and $S\subseteq G\setminus \{0\}$ a subset for which $\Sigma(S)$ is aperiodic.  Then 
\[
|\Sigma(S)|\ge \frac{|S|(|S|+2)}{4}+1 \, .
\]
\end{conjecture}

We remark also that similar results may be proved when $\Sigma(S)$ has a non-trivial period (stabiliser).  For a subset $X \subseteq G$ we let 
\[
K=K(X)=\{g\in G: X+g=X\}\, ,
\]
and refer to $K$ as the \emph{period} of $X$. In addition, $X$ will be called $H$\em-periodic \em whenever $H$ is a subgroup of $G$ contained in $K$. 

\begin{theorem}\label{thm:symperiod} Let $G$ be a finite abelian group, $S\subseteq G$ a symmetric subset and $K$ the period of $\Sigma(S)$.  Then 
\[
|\Sigma(S)|\ge \frac{|S\setminus K|(|S\setminus K|-2)}{4}+|K| \, .
\]
\end{theorem}

The outline of the article is as follows.  In Section~\ref{sec:deduce}, we show how Theorems~\ref{thm:sym} and ~\ref{thm:symperiod} may be deduced from Theorem~\ref{thm:main}.  The only tools we shall require in Section~\ref{sec:deduce} are Kneser's Addition Theorem and the so-called prehistoric lemma.  In Section~\ref{sec:tools}, we introduce the main tools and techniques that we shall require for the proof of Theorem~\ref{thm:main}.  Curiously a technique introduced by Erd\H os and Heilbronn~\cite{EH64} and sharpened by Olson~\cite{Olson68} remains at the heart of our proof.  The proof of Theorem~\ref{thm:main} appears in Section~\ref{sec:proof}.

We remark that an extension of Olson's result in $\mathbb{Z}_p$ has recently been obtained by one of the authors~\cite{Balandraud12}.

\section{Subset sums of symmetric sets: Theorems~\ref{thm:sym} and~\ref{thm:symperiod}}\label{sec:deduce}

In this section, we deduce Theorems~\ref{thm:sym} and~\ref{thm:symperiod} from Theorem~\ref{thm:main}.  We shall require Kneser's Addition Theorem, the prehistoric lemma and a simple observation concerning aperiodic sets.  As usual for subsets $X,Y\subseteq G$ we let $X+Y:=\{x+y:\, x\in X,y\in Y\}$.

\begin{theorem}[Kneser's Addition Theorem, \cite{Kn1,Kn2,NAT,TaoVu}] 
\label{thm:Kneser}
Let $X,Y$ be two subsets of a finite abelian group $G$, and let $H$ be the period of $X+Y$. Then 
$$|X+Y| \ge  |X + H| + |Y + H| - |H|.$$
\end{theorem}

We include a second statement in the prehistoric lemma which is an immediate consequence of the first and will be useful in many of our applications of the lemma.

\begin{lemma}[Prehistoric Lemma]   
\label{lem:prehistoric}
If $X, Y$ are two subsets of a finite abelian group $G$ and $|X| + |Y| > |G|$ then $X + Y = G$.  Furthermore, if $H\ssq G$ is a subgroup, $X\ssq Q, Y\ssq R$ are subsets of $H$-cosets $Q$ and $R$ and $|X|+|Y|>|H|$ then $X+Y=Q+R$.
\end{lemma}

\begin{obs} \label{obs:aperiodic}
If $\Sigma(S)$ is aperiodic and $T\subseteq S$, then $\Sigma(T)$ is aperiodic.  If $\Sigma(S)$ has period $K$ and $T\subseteq S$, then the set $\{Q\in G/K: \Sigma(T)\cap Q\neq \emptyset\}$ is aperiodic in $G/K$.
\end{obs}

Let us now deduce Theorem~\ref{thm:sym} from Theorem~\ref{thm:main}.

\begin{proof}[Proof of Theorem~\ref{thm:sym}]
We prove the first bound, the second bound follows with an identical proof except using property ($i'$) rather than ($i$) in the application of Theorem~\ref{thm:main}.
Let $G$ be a finite abelian group, and $S\subseteq G\setminus \{0\}$ a symmetric subset with $|S|\ge 4$ for which $\Sigma(S)$ is aperiodic.  If $S$ contains an element $x$ of order two then $\Sigma(\{x\})=\{0,x\}$ is not aperiodic, a contradiction, by the above observation.  Thus we may assume that $S$ contains no element of order two.  It follows that $S$ contains a subset $S_+$ of cardinality $|S_+|=|S|/2$ such that $S=S_+\cup (-S_+) $ and $S_+\cap (-S_+)=\emptyset$.

By applying Theorem~\ref{thm:main} to $S_+$ we obtain that either $|\Sigma(S_+)|\ge 3+|S_+|(|S_+|-1)/2$ or $S_+$ contains a non-empty subset $S'$ such that $|\Sigma(S')|>|\subgp{S'}|/2$.  In the first case we note that, by symmetry, the same bound also applies to $|\Sigma(-S_+)|$, and so by an application of Kneser's Addition Theorem (and using the fact that $\Sigma(S)=\Sigma(S_+)+\Sigma(-S_+)$ is aperiodic) we have that
\[
|\Sigma(S)|\ge |\Sigma(S_+)|+|\Sigma(-S_+)|-1\ge 2\left(\frac{(|S|/2)(|S|/2-1)}{2}+3\right )-1=\frac{|S|(|S|-2)}{4}+5\, ,
\]
as required.  In the second case we note that, by symmetry, $|\Sigma(-S')|>|\subgp{S'}|/2$, and so $\Sigma(S'),\Sigma(-S')\subseteq \subgp{S'}$ are such that $|\Sigma(S')|+|\Sigma(-S')|>|\subgp{S'}|$ and so
\[
\Sigma(S'\cup (-S'))=\Sigma(S')+\Sigma(-S')=\subgp{S'}
\]
by the prehistoric lemma.  However, this implies that $\Sigma(S'\cup (-S'))$ is not aperiodic, a contradiction, by the above observation, and the proof is complete.\end{proof}

We now prove Theorem~\ref{thm:symperiod}.

\begin{proof}[Proof of Theorem~\ref{thm:symperiod}]
Since $0 \in \Sigma(S)$, we readily have $\Sigma(S) \supseteq K$, so that $|\Sigma(S)| \ge |K|$. 
This yields the desired result if $S \setminus K = \emptyset$. 
Thus, we can assume that $S \setminus K \neq \emptyset$. 
Now, note that it suffices to prove the inequality for $T:=S \setminus K$.  
We associate to $T$ a sequence of subsets of $G/K$.  Let $k:=|K|$. 
We define the sets $T_1,\dots , T_k\subseteq G/K$ by
\[
T_i :=\{Q\in G/K: |T\cap Q|\ge i\}\qquad i=1,\dots ,k\, ,
\]
and write $l$ for the maximal $i$ for which $T_i$ is non-empty. Note that each of the sets $T_i:i=1,\dots ,l$ is symmetric. The key observation is that an element of $G$ belongs to $\Sigma(T)$ if and only if the coset of $K$ to which it belongs is an element of
\[
\Sigma(T_1)+\dots +\Sigma(T_l)\, .
\]
So that 
\[
|\Sigma(T)|=k|\Sigma(T_1)+\dots +\Sigma(T_l)|\ge k\left(|\Sigma(T_1)|+\dots +|\Sigma(T_l)|-(l-1)\right) \, ,
\]
where the inequality follows from Kneser's Addition Theorem together with the observation that $\Sigma(T_1)+\dots +\Sigma(T_l)$ is aperiodic in $G/K$ (this follows from the definition of $K$, as an element of a $K$-coset that leaves $\Sigma(T_1)+\dots +\Sigma(T_l)$ invariant under addition would also leave $\Sigma(T)$ invariant under addition). Thus, to complete the proof of the theorem, it suffices to prove that 
 \[
 |\Sigma(T_1)|+\dots +|\Sigma(T_l)|\ge 2l + \frac{|T|(|T|-2)}{4l}\, .
 \]
However, it follows immediately from the bound
\[
|\Sigma(T_i)|\ge\frac{|T_i|(|T_i|-2)}{4} +2\, ,
\]
which is a consequence of Theorem~\ref{thm:sym} (which may be applied since $T_i$ is symmetric and $|\Sigma(T_i)|$ is aperiodic in $G/K$, see Observation~\ref{obs:aperiodic}), and the convexity of the function $f(t)=t(t-2)$, that
\[
  |\Sigma(T_1)|+\dots +|\Sigma(T_l)|\ge 2l+\sum_{i=1}^{l}\frac{|T_i|(|T_i|-2)}{4}\ge 2l+\frac{1}{l}\frac{|T|(|T|-2)}{4}\, ,
  \]
  which completes the proof.
  \end{proof}

\section{Some tools and techniques}\label{sec:tools}

In this section, we present the tools and techniques on which we base our proof of Theorem~\ref{thm:main}.  Our approach is very similar in spirit to the approach of Olson~\cite{Olson68}.  His method, a refinement of that of Erd\H os and Heilbronn, is inductive.  However, rather than considering only a single base case (such as $|S|=1$), he proves the required bound directly for all arithmetic progressions, and these become the base cases of the inductive proof.  For the inductive step he may then assume that $S$ is not an arithmetic progression in which case (with some work) one may find an element $x\in S$ such that $|\Sigma(S)|-|\Sigma(S\setminus \{x\})|$ is large, and the proof is completed by applying the induction hypothesis to $S\setminus \{x\}$.

\medskip
In generalising Olson's approach we replace his dichotomy (whether or not $S$ is an arithmetic progression) with the dichotomy of whether or not the set $\hat{S}=S\cup \{0\}\cup (-S)$ is an arithmetic progression relative to a certain subgroup $H$ of $G$, where $H$, a subgroup chosen as a function of $\hat{S}$, is given by applying the following theorem of Hamidoune and Plagne~\cite[Theorem 2.1]{HP04} to $\hat{S}$. This result from critical pair theory, whose proof relies on the so-called \em isoperimetric method\em, may be seen as a generalisation of Vosper's Theorem to the general case of finite abelian groups. 
Before stating the result, we recall the following terminology. 
A subset $X$ of a finite abelian group $G$ is a \em Vosper subset \em in $G$ if for any $Y \subseteq G$, with $|Y| \ge 2$, the inequality
$$|X + Y| \ge \min(|G| - 1, |X| + |Y|)$$
holds. 
Notice that a Vosper subset with cardinality one cannot exist in a group with cardinality four or more. 
In what follows, we denote by $\phi$ the canonical homomorphism from $G$ to $G / H$. 

\begin{theorem}[Hamidoune-Plagne]
\label{thm:HP04}
Let $A$ be a generating subset of a finite abelian group $G$ such that $0 \in A$. 
Suppose also 
$$|A| \le \frac{|G|}{2}.$$
Then, there exists a subgroup $H$ of $G$ with
$$|A + H| < \min(|G|, |H| + |A|)$$
such that $\phi(A)$ is either an arithmetic progression or a Vosper subset in $G / H$.
\end{theorem}

We will also use the following theorem, proved recently by some of the authors~\cite{GGH}, concerning $k\wedge A:=\{a_1+\dots +a_k: a_i\in A \, \text{distinct}\}$.  We call a coset of an elementary $2$-subgroup of $G$ a \emph{$2$-coset}.

\begin{theorem}
\label{kwedgeA}
Let $A$ be a finite subset of an abelian group $G$, and let $1\le k\le |A|-1$.  Then 
\[
|k\wedge A|\ge |A|\, ,
\]
unless $k\in \{2,|A|-2\}$ and $A$ is $2$-coset, in which case $|k\wedge A|=|A|-1$.
\end{theorem}

In particular, if $H$ is a subgroup of $G$ and $A$ a subset of an $H$-coset such that $|H|/2<|A|\le |H|$ then
\begin{equation}\label{eq:kwedgeA}
|k\wedge A|\ge\min(|H|-1,|A|)\, .
\end{equation}

We complete the section by recalling some key results related to Olson's method.

\subsection{Olson's method}

Let $B \subseteq G$ and  $x \in G$. We write
$$\lambda_B(x) = |(B+x) \setminus {B}|.$$
An interesting feature of this number is that if $S \subseteq G$ and $B=\Sigma(S)$, then for all $x \in S$,
\begin{equation}
|\Sigma(S)| \ge |\Sigma({S \setminus \{x\}})|+ \lambda_B(x).
\label{eqolson}
\end{equation}
Some immediate properties of $\lambda_B$ are given in the following lemma.
\begin{lemma}[Olson, \cite{Olson68,Olson75}] 
Let $B$ and $C$ be non-empty subsets of a finite abelian group $G$ such that $0 \notin C$. 
Then, for all $x,y \in G$, we have
\label{Proprietes de lambda}
\begin{eqnarray}
\lambda_B(x) & = & \lambda_{G \setminus B}(x). \label{eq:complement}\\
\lambda_B(x) & = & \lambda_B(-x). \label{eq:-x}\\
\lambda_B(x+y) & \le & \lambda_B(x) + \lambda_B(y). \label{eq:x+y}\\
\sum _{x \in C} \lambda_B(x) & \ge & |B|(|C|-|B|+1). \label{eq:clique}
\end{eqnarray}
\end{lemma}
We will also use the following lemma, which states that one can always swap an element $x\in S$ for $-x$ without changing the number of subset sums. 
In addition, the resulting set of subset sums is aperiodic if and only if $\Sigma(S)$ is so.      
\begin{lemma}[Olson, \cite{Olson68,Olson75}]
\label{olsonrepl}
Let $S$ be a non-empty subset of $G \setminus \{ 0\}$. 
For any $x \in S$, one has  $|\Sigma((S \setminus \{x\})\cup \{-x\})|=|\Sigma(S)|$. Furthermore, $\Sigma((S \setminus \{x\})\cup \{-x\})$ is aperiodic if and only if $\Sigma(S)$ is so.
\end{lemma}

The main idea in Olson's method is to find conditions which guarantee the existence of an element $x \in S$ such that $\lambda_B(x)$ is large.

\begin{lemma}[Olson \cite{Olson75}]
\label{Minoration de lambda} 
Let $G$ be a finite abelian group and let $S$ be a generating subset of $G$ such that $0 \notin S$. 
Let $B$ be a subset of $G$ such that $|B| \le |G|/2$. Then there exists $x \in S$ such that
$$\lambda_B(x) \ge \min\left(\frac{|B| + 1}{2}, \frac{|S \cup (-S)| + 2}{4}\right).$$
\end{lemma}

\begin{proof}
This result follows, using  (\ref{eq:-x}), by applying Lemma $3.1$ of \cite{Olson75} to $S \cup (-S)$.
\end{proof}

We will also use the following lemma, which is a consequence of the main result in \cite{Hami98}.
\begin{lemma}[Hamidoune]
\label{Lemme 2S}
Let $S$ be a subset of a finite abelian group $G$ such that $S \cap (-S) = \emptyset$. 
Then
$$|\Sigma(S)| \ge 2|S|.$$
\end{lemma}
\begin{proof}
The proof follows easily by induction on $|S| \ge 1$. It trivially holds when $|S|=1$, so assume $|S| \ge 2$ and set $B = \Sigma(S)$.
If $|B| \ge |G|-1$, then since $|S| \le (|G|-1)/2$ we obtain $|B| \ge 2|S|$. Otherwise, we have $2 \le |B| \le |G|-2$. 
Now, by Lemma \ref{Minoration de lambda} applied to $B$ or $G \setminus B$, and using (\ref{eq:complement}), there exists $x \in S$ such that $\lambda_B(x) \ge 2$. 
By (\ref{eqolson}), $|B| \ge |\Sigma(S \setminus \{x\})| + 2 \ge 2|S|$.
\end{proof}

From these two results, we deduce the following useful lemma. 
\begin{lemma}
\label{Lemme 2S+1}
Let $S$ be a subset of a finite abelian group $G$ such that $S \cap (-S) = \emptyset$, $|\Sigma(S)| \le |G|/2$ and $|S| \ge 4$. 
Then
$$|\Sigma(S)| \ge 2|S| + 1.$$
\end{lemma}
\begin{proof}
Set $B = \Sigma(S)$. Since $|S| \ge 4$, we have $|B| \ge |S| +1 =5$. 
Now, by Lemma \ref{Minoration de lambda} applied to $B$, there exists $x \in S$ such that 
$$\lambda_B(x) \ge \min\left(3,5/2\right).$$
Thus, $\lambda_B(x) \ge 3$. Now, using Lemma \ref{Lemme 2S} and (\ref{eqolson}), $|B| \ge |\Sigma(S \setminus \{x\})| + 3 \ge 2(|S|-1)+3=2|S|+1$.
\end{proof}

\section{Proof of Theorem~\ref{thm:main}}\label{sec:proof}
Let $G$ be a finite abelian group and $S\subseteq G$ a subset such that $S\cap (-S)=\emptyset$ and $|S|\ge 2$. Without loss of generality we may assume that $S$ generates $G$, and we set $\hat{S}=S\cup \{0\}\cup (-S)$.  Our proof is inductive.  However, there is a certain class of sets for which an inductive proof is not appropriate.  Informally, these cases correspond to sets $S$ for which the structure of $\hat{S}$ resembles an arithmetic progression.  These cases are dealt with directly (see Proposition~\ref{prop:AP}).  With these cases as a base the theorem may then be proved by induction on $|S|$.

Given a generating subset $A$ of a finite abelian group $G$ such that $0 \in A$ and $|A|\le |G|/2$, we may apply the Hamidoune-Plagne Theorem (Theorem~\ref{thm:HP04}) to $A$ to obtain a subgroup $H$ of $G$ with the properties that $|A + H| < \min(|G|, |H| + |A|)$ and $\phi(A)$ is either an arithmetic progression or a Vosper subset in $G / H$ (where $\phi$ denotes the canonical homomorphism from $G$ to $G / H$).  In the case that $\phi(A)$ is an arithmetic progression we say that $A$ has an \emph{AP-representation}.  In the case that $\phi(A)$ is a Vosper set we say that $A$ has a \emph{Vosper-representation}.

We shall deduce Theorem~\ref{thm:main} from the following proposition and lemmas.

\begin{proposition}\label{prop:AP} Let $G$ be a finite abelian group, and let $S\subseteq G$ be a generating subset such that $S\cap (-S)=\emptyset$ and $|S|\ge 4$.  If $\hat{S}$ has an AP-representation, then one of the following holds.
\begin{itemize}
\item[($i$)] 
\[
|\Sigma(S)|\ge \frac{|S|(|S|+1)}{2}+1 \, .
\]
\item[$(ii)$] There is a non-empty subset $S'\subseteq S$ for which
\[
|\Sigma(S')|>\frac{|\subgp{S'}|}{2}\, .
\]
\end{itemize}
\end{proposition}

The following two lemmas make claims concerning $\max_{x\in S}\lambda_B (x)$ for subsets $S,B$ of a finite abelian group $G$.  These bounds, applied with $B=\Sigma(S)$, are precisely what is required for our inductive proof of Theorem~\ref{thm:main}.

\begin{lemma}\label{lem:VG} Let $G$ be a finite abelian group, and let $B,S$ be subsets of $G$ with $|B|=b\le |G|/2$ and $|S|=s\ge 3$.  Assume $S$ generates $G$, that $S\cap (-S)=\emptyset$ and that $\hat{S}$ has a Vosper-representation.  Then
$$\max_{x \in S} \lambda_B(x) > s - \frac{s(s-3)}{b}.$$
In particular, if $2b\ge s(s-3)$, then
$$\max_{x \in S} \lambda_B(x) \ge s-1.$$
\end{lemma}

\begin{lemma}\label{lem:VI}
Let $G$ be a finite abelian group of odd order, and let $B,S$ be subsets of $G$ with $|B|=b\le |G|/2$ and $|S|=s\ge 3$.  Assume $S$ generates $G$, that $S \cap (-S)=\emptyset$ and that $\hat{S}$ has a Vosper-representation. 
Let also $t$ be an integer, $1 \le t \le |G|-1$, and set
$$t=r(2s+2)+q,  \text{ where } -1 \le q \le 2s.$$
Then
$$\max_{x \in S} \lambda_B(x) \ge \frac{4(s+1)b(t-b+1)}{t(t+2s+6)+q(2s-q-2)}.$$
\end{lemma}

We now observe that Theorem~\ref{thm:main} is an immediate consequence of the above proposition and lemmas.  In fact the first part of Theorem~\ref{thm:main} may be deduced from Proposition~\ref{prop:AP} and Lemma~\ref{lem:VG}, while Lemma~\ref{lem:VI} is required for the stronger bound in the case $|G|$ is odd. 

\medskip
In the proof of Theorem~\ref{thm:main}, we will refer to $S$ as a \em valid subset \em of $G$ whenever
\[
|\Sigma(S')|\le \frac{|\subgp{S'}|}{2}
\]
for all non-empty subsets $S'\ssq S$.

\begin{proof}[Proof of Theorem~\ref{thm:main}]
Let $S$ be a generating subset of $G$ such that $S\cap (-S)=\emptyset$ and $|S|\ge 2$. 
We begin by proving the first part of Theorem~\ref{thm:main}. Without loss of generality, we may also assume that $S$ is a valid subset of $G$, 
else property $(ii)$ holds and the proof is complete.
Now, setting $\hat{S}=S\cup \{0\}\cup (-S)$, Lemma~\ref{Lemme 2S+1} yields 
\[ |\hat{S}|=2|S|+1\le |\Sigma(S)|\le |G|/2.\]
Thus, it follows from Theorem~\ref{thm:HP04} that $\hat{S}$ has either an AP-representation or a Vosper-representation.
We must prove that 
\[
|\Sigma(S)|\ge\frac{|S|(|S|-1)}{2}+3.
\]
The base cases that $|S|=2,3$ may be checked by hand, while the base case that $\hat{S}$ has an AP-representation follows from Proposition~\ref{prop:AP}.  We now proceed to the induction step.  

Assume $|S|=s \ge 4$ and that
$\hat{S}$ has a Vosper-representation.  
Let $B=\Sigma(S)$ and $b=|B|$.  Since $S$ is a valid subset of $G$, one has $b\le |G|/2$. We prove the required bound $b\ge 3+s(s-1)/2$ by considering $b=|\Sigma(S)|\ge |\Sigma(S\setminus \{x\})|+\lambda_B (x)$ for an appropriately chosen $x\in S$.  An initial lower bound on $b$ may be obtained by selecting an arbitrary element $x\in S$ and using that 
\[
b= |\Sigma(S)| \ge |\Sigma(S \setminus \{x\})| \ge 3 + \frac{(s-2)(s-1)}{2},
\]
where the final inequality follows from the induction hypothesis.  It follows that
\[
2b \ge 6+(s-2)(s-1) > s(s-3).
\]
Now, by Lemma~\ref{lem:VG} there is an element $x \in S$ with $\lambda_B(x) \ge s-1$.
Thus
$$|\Sigma(S)| \ge |\Sigma(S \setminus \{x\})|+\lambda_B(x) \ge 3 + \frac{(s-2)(s-1)}{2} + (s-1) = 3 + \frac{s(s-1)}{2},$$ 
as required.

\medskip
For the second part of Theorem~\ref{thm:main}, the stronger bound in the case that $|G|$ is odd, we proceed by induction with the same base cases.  
For the induction step, assume $|S|=s \ge 4$ and that $\hat{S}$ has a Vosper-representation.   
One can distinguish the following two cases.

\medskip
\textbf{Case I. There is an element $x\in S$ such that $\subgp{S\setminus \{x\}}$ is a proper subgroup of $\subgp{S}$.}

In this case, the induction step is easy.  We simply use that $\xi(S\setminus \{x\})\ge 0$ to obtain 
\[
|\Sigma(S)|=2|\Sigma(S\setminus \{x\})| \ge (s-1)s\ge \frac{s(s+1)}{2}+1\, .
\]

\textbf{Case II. $\subgp{S\setminus \{x\}}=\subgp{S}$ for all $x\in S$.}

Let $B=\Sigma(S)$ and $b=|B|$.  Since $S$ is a valid subset of $G$, one has $b\le |G|/2$. 
Arguing as in the first part of the proof, there exists an element $x\in S$ such that $\lambda_B (x) \ge s-1$.  It follows, by the induction hypothesis, that 
\begin{align*}
|\Sigma(S)|\, & \ge \, |\Sigma(S\setminus \{x\})|+(s-1)\\
& \ge \, \frac{s(s-1)}{2}+\xi(S\setminus \{x\})+s-1\\
& = \, \frac{s(s+1)}{2}+\xi(S\setminus \{x\})-1\, .
\end{align*}
In the special case that $\xi(S\setminus \{x\})=1$ and $\xi(S)=0$ this bound is sufficient to complete the proof.  If $\xi(S\setminus \{x\})=0$, it follows that 
\[
2s^2-s-1 = 2(s-1)^2+3(s-1) > 2|\subgp{S\setminus \{x\}}|+5 =2|\subgp{S}|+5\, ,
\]
and so
\[
|\Sigma(S)|\ge \frac{s(s+1)}{2} - 1=\frac{2s^2+2s-4}{4} >\frac{2|\subgp{S}|}{4}=\frac{|G|}{2} \, ,
\]
a contradiction, since $S$ is a valid subset of $G$.  Thus, the only remaining case is that $\xi(S\setminus \{x\})=\xi(S)=1$.  In particular we can assume that
\[
\begin{cases}
s^2+s-2 \le |G|-1 & \quad \text{if $s$ is even}\\
s^2 +\frac{3}{2}s-\frac{7}{2} \le |G|-1 & \quad \text{if $s$ is odd}
\end{cases}
\]
Since we have that $b\ge s(s+1)/2$ and the proof is completed when we prove $b\ge 1+ s(s+1)/2$ we may assume for contradiction that $b=s(s+1)/2$.  However, one may now apply Lemma~\ref{lem:VI}, with
\[
t=
\begin{cases}
s^2+s-2&\quad \text{if $s$ is even}\\
s^2 +\frac{3}{2}s-\frac{7}{2}&\quad \text{if $s$ is odd}
\end{cases}
\]
and 
\[
q=
\begin{cases}
2s & \quad \text{if $s$ is even}\\
\frac{3}{2}s-\frac{5}{2}& \quad \text{if $s$ is odd}
\end{cases}
\]
to obtain that $\max_{x\in S}\lambda_B (x)> s-1$.  In particular, there exists $x\in S$ with $\lambda_B (x) \ge s$ and so
\[
|\Sigma(S)| \ge |\Sigma(S\setminus \{x\})|+s\ge \frac{s(s-1)}{2}+1+s =\frac{s(s+1)}{2}+1\, ,
\]
as required.
\end{proof}

We prove Proposition~\ref{prop:AP} in Section~\ref{sec:APcase} and Lemmas~\ref{lem:VG} and \ref{lem:VI} in Sections~\ref{sec:VG} and ~\ref{sec:VI} respectively.

\subsection{The case that $\hat{S}$ has an AP-representation: A proof of Proposition~\ref{prop:AP}}\label{sec:APcase}
Let $G$ be a finite abelian group and $S\subseteq G$ a generating subset such that $S\cap (-S)=\emptyset$, $|S|\ge 4$ and $\hat{S}$ has an AP-representation.
Let $H$ be a subgroup of $G$ with
\begin{equation}\label{APdef}
|\hat{S} + H| < \min(|G|, |\hat{S}| + |H|)
\end{equation}
and with $\phi(\hat{S})$ being an arithmetic progression in $G / H$.  
We may also assume throughout the proof that
\begin{equation}\label{nopropii}
|\Sigma(S')|\le \frac{|\subgp{S'}|}{2}
\end{equation}
for all non-empty subsets $S'\ssq S$ (i.e., $S$ is a valid subset of $G$), else property $(ii)$ of Proposition~\ref{prop:AP} holds and the proof is complete.

\medskip
Now, since $S$ is a generating subset of $G$, it follows that $G / H$ is a cyclic group. Thus, we may write $G/H$, the group of $H$-cosets in $G$, as follows
$$G / H \simeq \mathbb{Z}/m\mathbb{Z} \simeq \left\{Q_i : i=0,\dots,m-1\right\},$$
and we may assume, without loss of generality, that $\phi(\hat{S})$ has difference $Q_1$, so that $\phi(\hat{S})=\{Q_{-v},\dots ,Q_{v}\}$, for some $v\ge 1$.
We consider the partition
$$\hat{S}=\hat{S}_{-v} \cup \dots \cup \hat{S}_{-1} \cup \hat{S}_0 \cup \hat{S}_1 \cup \dots \cup \hat{S}_v,$$ 
where $\hat{S}_i=\hat{S}\cap Q_i$ for all $i \in \{-v,\dots,v\}$. 
Note that, by the symmetry of $\hat{S}$, we have $\hat{S}_{-i}=-\hat{S}_{i}$ for all $i \in \{-v,\dots,v\}$. 
Since Lemma~\ref{olsonrepl} allows us to swap an element $x\in S$ for $-x$ we may suppose that
\[
S=S_0 \cup S_1 \cup \dots \cup S_v\, ,
\]
where $S_i=\hat{S}_i$ for all $i \in \{1,\dots,v\}$ and $|\hat{S}_0|=2|S_0|+1$. 

We use the following notation:
\[
t:= |S_0|\qquad u:=\sum_{i=1}^{v}|Q_i\setminus S_i|\qquad \text{and}\qquad \ell:=\sum_{i=1}^{v}i|S_i|\, ,
\]
and write $h$ for $|H|$.  Note that, in this notation,
\[
|S|=vh+t-u\, ,
\]
and, by Lemma~\ref{Lemme 2S},
\[
|\Sigma(S_0)|\ge 2t\, .
\]

We now establish the following claims.

\bigskip
\textbf{Claim I.} $t\le h/4$.

\begin{proof}
If $|S_0|=t> h/4$, then $|\Sigma(S_0)|\ge 2t >h/2\ge |\subgp{S_0}|/{2}$, contradicting \eqref{nopropii}.
\end{proof}

\textbf{Claim II.} $u\le t$.

\begin{proof}
Since $|\hat{S}|=2|S|+1=2vh+2t-2u+1$, and $|\hat{S}+H|=(2v+1)h$, it follows from ~\eqref{APdef} that
\[
2t-2u+1>0\, ,
\]
and the required bound follows.
\end{proof}

\textbf{Claim III.} $\ell\ge hv(v+1)/2\,  -\, uv$.

\begin{proof}
Since the cardinalities $|S_1|,\dots ,|S_v|$ obey $0\le |S_i|\le h$ and $\sum_{i=1}^{v}|S_i|=vh-u$, the sum $\ell=\sum_{i=1}^{v}i|S_i|$ is minimised by taking $|S_i|=h$ for $i=1,\dots ,v-1$ and $|S_v|=h-u$, and in this case one obtains $\ell= hv(v+1)/2\,  -\, uv$.
\end{proof}

We now prove the key lemma from which we shall deduce our bound on $|\Sigma(S)|$.  The idea behind the proof is that $\Sigma(S)$ should contain all of the elements of all of the cosets $Q_1,\dots ,Q_{\ell-1}$ together with a few more elements in the case $t>0$.  

\begin{lemma}\label{lem:manysums} Let $S, h, \ell, v$ and $t$ be as defined above.  Then $|\Sigma(S)|\ge (\ell -1)h+4t$.\end{lemma}

\begin{proof} We prove the lemma under the assumption $\ell<|G/H|$, in which case the cosets $Q_0,\dots ,Q_{\ell}$ are disjoint.  
Since it may be easily verified (by a similar approach) that $|\Sigma(S)|>|G|/2$, contradicting~\eqref{nopropii}, in the case that $\ell\ge |G/H|$ we may safely restrict to this case.

\medskip
We consider first the special case that $v=1$ and $t=0$. It suffices to prove that $\Sigma(S)\supseteq Q_j$ for $j\in \{1,\dots ,h-1\}\setminus \{2,h-2\}$, $|\Sigma(S)\cap Q_j|\ge h-1$ for $j\in \{2,h-2\}$ and $|\Sigma(S)\cap Q_j|=1$ for $j\in \{0,\ell\}=\{0,h\}$.  To prove these bounds we note that $\Sigma(S)\cap Q_j\supseteq j\wedge S$, and the various claimed bounds are either trivial or follow from Theorem~\ref{kwedgeA}.

\medskip
For the remaining cases we claim that 
\[
\Sigma(S)\supseteq \bigcup_{j=1}^{\ell-1} Q_j
\]
and $|\Sigma(S) \cap Q_j|\ge 2t$ for $j\in \{0,\ell\}$. 
It is immediate, since $|\Sigma(S_0)|\ge 2t$, that $|\Sigma(S) \cap Q_j|\ge 2t$ for $j\in \{0,\ell\}$.  The proof that $\Sigma(S)\supseteq Q_j$ for all $j\in \{1,\dots ,\ell-1\}$ proceeds slightly differently in the cases $t=0$ and $t>0$.  We note that this fact is trivial if $h=1$, so we may assume that $h\ge 2$.

\medskip
If $t>0$ then we recall that $|\Sigma(S_0)|\ge 2t$ and that $u\le t$ (Claim II).  We also have that $h\ge 4$ (Claim I) and $|S_i|\ge \frac{3}{4}h>2$ for all $i$ (Claims I and II).  Fix $j\in \{1,\dots ,\ell-1\}$ and note that, by the definition of $\ell$, there exists a sequence $k_1,\dots , k_v$ such that 
\[
\sum_{i=1}^{v}i k_i= j
\]
where $0\le k_i\le |S_i|$ for all $i$, and $0<k_{i_0}<|S_{i_0}|$ for some $i_0\in \{1,\dots ,v\}$.
The claim that $Q_j\subseteq \Sigma(S)$ now follows from the prehistoric lemma and the observation that 
\[
\Sigma(S)\cap Q_j \supseteq (k_1\wedge S_1) + \dots +(k_v\wedge S_v) +\Sigma(S_0)\, ,
\]
since 
\[
|(k_1\wedge S_1) + \dots +(k_v\wedge S_v)|\ge |k_{i_0}\wedge S_{i_0}|\ge \min(h-1,h-u)
\]
(by \eqref{eq:kwedgeA}) and
\[
|\Sigma(S_0)|\ge 2t\ge \max(2,2u)
\]
sum to more than $h=|H|=|Q_j|$.

The argument in the case that $t=0$ is similar, except on this occasion we use that $j$ may be expressed as
\[
\sum_{i=1}^{v}i k_i= j
\]
where $0\le k_i\le h$ for all $i$, and either $0<k_{i}<h$ for two values of $i \in \{1,\dots ,v\}$ or $k_{i_0}\in \{1,h-1\}$ for some $i_0$.  In the latter case we observe immediately that 
\[
|\Sigma(S)\cap Q_j|\ge |(k_1\wedge S_1) + \dots +(k_v\wedge S_v)|\ge |k_{i_0}\wedge S_{i_0}|=|S_{i_0}|=h\, ,
\]
which implies that $Q_j\subseteq \Sigma(S)$ as required.  In the former case we simply use the prehistoric lemma applied to the two sets $k_i\wedge S_i$ for which $0<k_{i}<h$ and Theorem~\ref{kwedgeA}, as above to obtain $Q_j\subseteq \Sigma(S)$.\end{proof}

We may now read out the bound 
\[
|\Sigma(S)|\ge \frac{|S|(|S|+1)}{2}+1\, ,
\]
completing the proof of Proposition~\ref{prop:AP}.  We prove that the quantity $\Delta:=|\Sigma(S)|-|S|(|S|+1)/2$ satisfies $2\Delta\ge 2$, as required.
By combining Lemma~\ref{lem:manysums} with Claim III we obtain that
\begin{align*}
2\Delta  \, & = \, 2|\Sigma(S)|-|S|(|S|+1)\\
&\ge \,  (hv(v+1)\,  -\, 2uv - 2)h  + 8t \, -\,  (hv+t-u)(hv+t-u+1)\\
&= \, h^2 v^2 +h^2 v -2uhv -2h + 8t -h^2v^2 -(2t-2u+1)hv -(t-u)(t-u+1) \\
&= \, h^2 v -2h +8t - (2t+1)hv -(t-u)(t-u+1)\, ,
\end{align*}
where the final line is obtained simply by canceling terms.  We first complete the proof in the case that $h\ge 4$.  We shall deal with the special cases $h\in \{1,2,3\}$ separately.  Since $h\ge 4$ we have that $vh\ge 4$ and so $8t-(2t+1)vh$ is decreasing in $t$.  Note also that, since $t\ge u$, the term $-(t-u)(t-u+1)$ is also decreasing in $t$.  Thus, the final expression above is decreasing in $t$.  Since $t\le h/4$ (by Claim I), and $u\ge 0$, we have that
\begin{align*}
2\Delta \, &\ge\, h^2 v -\frac{h^2 v}{2}- hv - \frac{h^2}{16}- \frac{h}{4}\\
& \ge h^2 v\left(1-\frac{1}{2}-\frac{1}{4}-\frac{1}{16}-\frac{1}{16}\right)\\
& = \frac{h^2 v}{8}\\
& \ge 2\, .
\end{align*}

For the special cases $h\in \{1,2,3\}$ we have that $t=0$ by Claim I, and $u=0$ by Claim II.  One may easily check the result by hand for the case that $h\in \{1,2,3\}$ and $v=1$.  So let us assume that $v\ge 2$.  Proceeding as in the proof of Lemma~\ref{lem:manysums} and using the fact that $|\Sigma(S)\cap Q_j|\ge 1$ for $j\in \{0,\ell\}$ one obtains that $|\Sigma(S)|\ge (\ell-1)h+2$.  Combining this with the fact that $t=u=0$ and Claim III, we obtain that
\begin{align*}
2\Delta\, & \ge\,  h^2 v^2+ h^2 v-2h+4 -h^2 v^2 -hv\\
& = \, h^2 v - 2h +4 - hv
& = \,
\begin{cases}
2 & \text{ if } h=1,\\ 
2v & \text{ if } h=2, \\ 
6v-2 & \text{ if } h=3.  
\end{cases}
\end{align*}
Since each of these values is at least $2$ we obtain that $|\Sigma(S)|\ge 1+|S|(|S|+1)/2$, thus completing the proof of Proposition~\ref{prop:AP}.

\subsection{The case that $\hat{S}$ has a Vosper-representation: A proof of Lemma~\ref{lem:VG}}  \label{sec:VG}
In this section, we prove Lemma~\ref{lem:VG}.  That is, we show that if $B,S$ are subsets of a finite abelian group $G$ with $|B|=b\le |G|/2$ and $|S|=s\ge 3$, and we have the additional properties that $S$ generates $G$, that $S\cap (-S)=\emptyset$ and that $\hat{S}$ has a Vosper-representation, then
$$\max_{x \in S} \lambda_B(x) > s - \frac{s(s-3)}{b}.$$
This is sufficient since the second claim of Lemma~\ref{lem:VG} is an immediate consequence of the first. 

Our proof proceeds via demonstrating a certain rate of expansion of the sets $j\hat{S}$, when $S$ is as above.  We say that a subset $A$ of $G$ is \em faithful \em if, for every integer $j \ge 1$, one has 
$$|j\hat{A}| \ge \min\left(|G|,j(|\hat{A}|-1)+1\right).$$

It is clear that the required result follows immediately once we establish the following two lemmas.

\begin{lemma}\label{lem:faithful} Let $G$ be a finite abelian group, and let $S\ssq G$.  Assume $S$ generates $G$ and that $\hat{S}$ has a Vosper-representation.  Then $S$ is faithful.\end{lemma}

\begin{lemma}
\label{lemme lambda general} 
Let $G$ be a finite abelian group, and let $B,S$ be subsets of $G$ with $|B|=b\le |G|/2$ and $|S|=s\ge 3$.  Assume that $S\cap (-S)=\emptyset$ and $S$ is faithful.  Then
$$\max_{x \in S} \lambda_B(x) > s - \frac{s(s-3)}{b}.$$
\end{lemma}

We begin with some initial observations that we shall use in our proof of Lemma~\ref{lem:faithful}.  Let $G$ be a finite abelian group and $S\subseteq G$ a generating subset with $|S|=s\ge 3$ and such that $\hat{S}$ has a Vosper-representation.  Recall that $\hat{S}=S\cup \{0\}\cup (-S)$ and let $H$ be a subgroup of $G$ such that
\begin{equation}\label{def}
|\hat{S} + H| < \min(|G|, |\hat{S}| + |H|)
\end{equation}
and with $\phi(\hat{S})$ being a Vosper subset in $G / H$.

We first establish a basic lemma.

\begin{lemma}\label{lem:2hats} $2 \hat{S}$ is $H$-periodic.\end{lemma}

As we shall see, Lemma~\ref{lem:2hats} is an elementary consequence of \eqref{def}.  We choose to write $\hat{S}_Q$ for $\hat{S}\cap Q$ for each coset $Q$ of $H$.  So that
\[
\hat{S}=\bigcup_{Q\in \phi(\hat{S})} \hat{S}_Q\, .
\]
The following two facts, together with the prehistoric lemma, are all that we require to deduce Lemma~\ref{lem:2hats}.  Equation \eqref{def} is used in the proof of each of the facts.

\medskip
\textbf{Fact 1.} If $Q,R \in \phi(\hat{S})$ are two $H$-cosets with $Q \neq R$, then 
\begin{eqnarray*}
\left|\hat{S}_Q\right| + \left|\hat{S}_R\right| & \ge & 2|H|-|(\hat{S} + H)\setminus \hat{S}|\\ 
                                                    & > & \left|H\right| \, .
\end{eqnarray*} 

\medskip
\textbf{Fact 2.} If $Q\in \phi(\hat{S})$ is an $H$-coset with $Q \neq H$, then Fact 1 implies
\begin{eqnarray*}
\left|\hat{S}_Q\right| + \left|\hat{S}_Q\right| &    =   & \left|\hat{S}_Q\right| + \left|\hat{S}_{-Q}\right|\\ 
                                                        &  > & |H|\, .
\end{eqnarray*}

\begin{proof}[Proof of Lemma \ref{lem:2hats}] By Facts 1 and 2 we have that
\[
|\hat{S}_Q|+|\hat{S}_R|>|H|
\]
for all pairs $Q,R\in \phi(\hat{S})$ other than $(Q,R)=(H,H)$.  It follows by the prehistoric lemma that
\[
\hat{S}_Q+\hat{S}_R=\hat{S}_Q +\hat{S}_R +H
\]
for all pairs $Q,R\in \phi(\hat{S})$ other than $(Q,R)=(H,H)$.  Since $H$ may be represented by $Q+(-Q)$ for any $Q\in \phi(\hat{S})\setminus \{H\}$, this establishes that $2\hat{S}=2\hat{S}+H$, i.e., $2\hat{S}$ is $H$-periodic, as required.\end{proof}

An immediate consequence of Lemma~\ref{lem:2hats} is that $j\hat{S}$ is $H$-periodic for all $j\ge 2$.  It then follows that $j\hat{S}$ consists precisely of all elements that belong to $H$-cosets $Q\in j\phi(\hat{S})$.  In particular
\begin{equation}\label{eq:hfois}
|j\hat{S}|=|H||j\phi(\hat{S})|\qquad \text{for all} \, j\ge 2\, .
\end{equation}

\medskip
Now, we prove that $S$ is faithful.

\begin{proof}[Proof of Lemma~\ref{lem:faithful}]
Let $t \in \mathbb{N}$ be the greatest integer such that $t\hat{S} \neq G$.  It is immediate that $S$ is faithful in the case that $t=1$.  
In the case that $t\ge 2$ we shall in fact prove that
\[
|j\hat{S}|\ge 
\begin{cases} 
j|\hat{S}| &\qquad \text{for } j=1,\dots ,t-1\\
j|\hat{S}|-1 & \qquad \text{for } j=t \\
|G| &\qquad \text{for } j>t \, 
\end{cases} 
\]
which clearly implies that $S$ is faithful.
The claimed bound is trivial for $j>t$ (by the definition of $t$).  For $j=1,\dots ,t-1$ we note that the required bounds follow directly from~\eqref{eq:hfois} and the bounds
\[
|j\phi(\hat{S})|\ge j|\phi(\hat{S})|\qquad j=1,\dots ,t-1\, ,
\]
which we now prove by induction on $j$.  The base case $j=1$ is trivial.  For $j=2,\dots ,t-1$, we obtain by the Vosper property of $\phi(\hat{S})$ and the induction hypothesis that
\[
|j\phi(\hat{S})|=|(j-1)\phi(\hat{S})+\phi(\hat{S})|\ge \min(|G/H|-1,j|\phi(\hat{S})|)\, .
\]
If the bound $|j\phi(\hat{S})|\ge j|\phi(\hat{S})|$ is obtained then the proof of the induction step is complete, so we may assume for contradiction that $|j\phi(\hat{S})|=|G/H|-1$.  However, in this case $|j\phi(\hat{S})|+|\phi(\hat{S})|>|G/H|$, and so $(j+1)\phi(\hat{S})=G/H$ by the prehistoric lemma.  It follows that $(j+1)\hat{S}=G$, a contradiction since $j+1\le t$.  For the remaining case that $j=t$ we use the Vosper property of $\phi(\hat{S})$ and the result $|(t-1)\phi(\hat{S})|\ge (t-1)|\phi(\hat{S})|$ obtained above to give that
\begin{equation}\label{eq:tphi}
|t\phi(\hat{S})|\ge |(t-1)\phi(\hat{S})+\phi(\hat{S})|\ge \min(|G/H|-1,t|\phi(\hat{S})|)\, .
\end{equation}
If $t|\phi(\hat{S})|\le |G/H|-1$ then the minimum is attained at $t|\phi(\hat{S})|$ and the claimed bound is proved.  If $t|\phi(\hat{S})|>|G/H|$ then $|\phi(\hat{S})|+|(t-1)\phi(\hat{S})|>|G/H|$, and so $t\phi(\hat{S})=G/H$ by the prehistoric lemma, and $t\hat{S}=G$, a contradiction.  In the remaining case we have $|G/H|-1<t|\phi(\hat{S})|\le |G/H|$, and so~\eqref{eq:tphi} gives us that $|t\phi(\hat{S})|\ge t|\phi(\hat{S})|-1$.  Combining this fact with the observation $|H||\phi(\hat{S})|\ge |\hat{S}|+(|H|-1)/2$ (which follows from Claim I of Section~\ref{sec:APcase}, for example) we obtain that 
\begin{eqnarray*}
|t\hat{S}|  &   =    & |H||t\phi(\hat{S})|\\
                 & \ge   & |H|(t|\phi(\hat{S})|-1)\\
                 & \ge & t|\hat{S}|+t\left(\frac{|H|-1}{2}\right)-|H|\\
                 & \ge & t|\hat{S}| -1.
\end{eqnarray*}
\end{proof}

Having established that $S$ is faithful we now prove that this is sufficient to guarantee the required bound on $\max_{x\in S}\lambda_B(x)$.  The proof follows (more or less step by step) the proof of \cite[Lemma 3.1]{HLS08}.

\begin{proof} [Proof of Lemma~\ref{lemme lambda general}]
We write
$$\alpha=\displaystyle\max_{x \in S} \lambda_B(x)\, ,$$ 
and note that in fact $\lambda_B (x)\le \alpha$ for all $x\in \hat{S}$.
Let $t \le |G| - 1$ be a positive integer and set 
$$t=2rs+q, \quad \text{ where } 0 \le q\le 2s-1.$$
Since $S$ is faithful the bounds $|j\hat{S}\setminus \{0\}|\ge \min(|G|-1,2js) \ge 2js$ for $j=1,\dots , r$, and $|(r+1)\hat{S}\setminus \{0\}|\ge \min(|G|-1,2(r+1)s)\ge t$ hold.  Hence one may select a sequence of disjoint sets $C_j\ssq G\setminus \{0\}\, :j=1,\dots ,r+1$ such that $C_j\ssq j\hat{S}$ for each $j=1,\dots ,r+1$, and with $|C_j|=2s\, :j=1,\dots ,r\, ,\, |C_{r+1}|=q$.  Set $C=\bigcup_{j=1}^{r+1}C_j$, and note that $C\ssq G\setminus\{0\}$ has cardinality $t=2rs+q$.  Our proof of the lemma proceeds via proving upper and lower bounds on the quantity $\sum_{c\in C}\lambda_B (c)$.

The lower bound on $\sum_{c\in C}\lambda_B (c)$ is given immediately by Lemma~\ref{Proprietes de lambda}:
\[
\sum_{c\in C}\lambda_B (c)\ge |C||B|-|B|^2+|B|\, =tb-b^2+b\, .
\]

For the upper bound on $\sum_{c\in C}\lambda_B (x)$ we use the sub-additivity of $\lambda_B (x)$ ensured by Lemma~\ref{Proprietes de lambda}.  Each element $c\in C_j\ssq j\hat{S}$ may be expressed as a sum 
\[
c = x_1 + \dots + x_j
\]
where $x_1,\dots,x_j$ are (not necessarily distinct) elements of $\hat{S}$, and so, by the sub-additivity of $\lambda_B (x)$ (Lemma~\ref{Proprietes de lambda}), we have $\lambda_B (c)\le \lambda_B (x_1) + \dots + \lambda_B (x_j) \le j\alpha$.  It follows that
\[
\sum_{c\in C_j} \lambda_B (c)\le |C_j|j\alpha \qquad j=1,\dots ,r+1\, ,
\]
and 
\begin{eqnarray*}
 \displaystyle\sum_{c \in C}\lambda_B(c) & \le & \sum_{j=1}^{r+1}|C_j|j\alpha\\
                                  & = & \alpha\sum_{j=1}^{r}2js +\alpha q(r+1)\\ 
                                                                          &   =  & \alpha(r+1)(rs+q) \\
                                                                          &   =  & \frac{\alpha(t-q+2s)(t+q)}{4s}\\                      
                                                                          &   \le  & \frac{\alpha(t+s)^2}{4s}\, ,
\end{eqnarray*}
where the final inequality follows since the penultimate expression is maximised when $q=s$.

Combining our bound on $\sum_{c \in C}\lambda_B(c)$ yields the inequality
\[
\alpha \ge \frac{4sb(t-b+1)}{(t+s)^2}\, .
\]
In particular, since $2b-3 \le |G|-1$, we may set $t=2b-3$.  It follows that
\begin{eqnarray*}
\alpha & \ge  & \frac{4sb(b-2)}{(2b-3+s)^2}\\
            &  \ge &  s\left(\frac{b-2}{b}\right)\left(1-\frac{s-3}{b}\right)\\                        
            &   >   & s - \frac{s(s-3)}{b},
\end{eqnarray*}
where we have used $s\ge 3.$
\end{proof}

\subsection{The stronger bound in the case that $|G|$ is odd: A proof of Lemma~\ref{lem:VI}}\label{sec:VI}

Lemma~\ref{lem:VI} is effectively a strengthening of Lemma~\ref{lem:VG} established in the previous section, so it should not be surprising that the proof has many similarities to that given above.  Proving stronger bounds on the cardinalities $|j\hat{S}|:j\ge 2$ (in fact bounds identical to those proved by Olson in $\mathbb{Z}_p$) is an essential improvement.  The proof of Lemma~\ref{lem:VI} is completed by proving the following two lemmas.  A subset $A$ of $G$ will be called \em super faithful \em if, for every integer $j \ge 1$, one has 
$$|j\hat{A}| \ge \min\left(|G|,j(|\hat{A}|+1) -1\right).$$

\begin{lemma}\label{lem:superfaithful} Let $G$ be a finite abelian group of odd order, and let $S\ssq G$.  Assume $S$ generates $G$ and that $\hat{S}$ has a Vosper-representation.  Then $S$ is super faithful.\end{lemma}

\begin{lemma}
\label{lemme lambda} 
Let $G$ be a finite abelian group of odd order, and let $B,S$ be subsets of $G$ with $|B|=b\le |G|/2$ and $|S|=s\ge 3$.  Assume that $S\cap (-S)=\emptyset$ and $S$ is super faithful.  Let also $t$ be an integer, $1 \le t \le |G|-1$, and set 
$$t=r(2s+2)+q,  \text{ where } -1 \le q \le 2s.$$
Then, there exists $x \in S$ such that
$$\lambda_B(x) \ge \frac{4(s+1)b(t-b+1)}{t(t+2s+6)+q(2s-q-2)}.$$
\end{lemma}

We proceed directly to the proof of the lemmas.

\begin{proof}[Proof of Lemma~\ref{lem:superfaithful}]
Let $G$ be a finite abelian group of odd order and let $S\ssq G$ be a generating subset with $|S|=s\ge 3$ and such that $\hat{S}$ has a Vosper-representation.  Let $H$ be a subgroup of $G$ such that
\begin{equation}\label{defI}
|\hat{S} + H| < \min(|G|, |\hat{S}| + |H|)
\end{equation}
and with $\phi(\hat{S})$ being a Vosper subset in $G / H$.
Let $t \in \mathbb{N}$ be the greatest integer such that $t\hat{S} \neq G$.  It is immediate that $S$ is super faithful in the case that $t=1$.  Thus we may assume that $t\ge 2$.  Using the fact, established in Section~\ref{sec:VG}, that $j\hat{S}$ is $H$-periodic for all $j\ge 2$ it suffices to prove that
\begin{equation}\label{jphihats}
|j\phi(\hat{S})| \ge j (|\phi(\hat{S})|+1)-1
\end{equation}
for $j=2,\dots,t$, as this implies
\[
|j\hat{S}|=|H||j\phi(\hat{S})|\ge j|\phi(\hat{S})||H| + (j-1)|H|\ge j|\hat{S}|+(j-1)\, .
\] 
We prove that \eqref{jphihats} holds by induction on $j$, using the Vosper property of $\phi(\hat{S})$ and a parity trick of Olson.
Note that \eqref{jphihats} trivially holds for $j=1$. 
For $j=2,\dots ,t$, we obtain by the Vosper property of $\phi(\hat{S})$ that
$$|j\phi(\hat{S})| \ge \min\left(|G/H|-1, |(j-1)\phi(\hat{S})| + |\phi(\hat{S})|\right).$$

\medskip
Since $|G|$ is odd, so is $|G/H|$. In addition, the fact that $j\phi(\hat{S})$ is symmetric and contains $0$ implies that $|j\phi(\hat{S})|$ is odd. 
Thus, $|j\phi(\hat{S})| \ge |G/H|-1$ cannot occur, otherwise we would have $|j\phi(\hat{S})| \ge |G/H|$, so that $j\phi(\hat{S}) = G/H$, which implies $j\hat{S} = G$, a contradiction. 

\medskip
Therefore, 
$$|j\phi(\hat{S})| \ge |(j-1)\phi(\hat{S})| + |\phi(\hat{S})|.$$ 
Then, the induction hypothesis, and the very same argument of parity again ($|j\phi(\hat{S})|$ is odd), yields
\begin{eqnarray*}
|j\phi(\hat{S})| & \ge & |(j-1)\phi(\hat{S})| + |\phi(\hat{S})| + 1\\
                         & \ge & (j-1)\left(|\phi(\hat{S})| + 1\right) -1 + |\phi(\hat{S})| + 1\\
                         &    =    & j\left(|\phi(\hat{S})| + 1\right) -1,
\end{eqnarray*}
as required.
\end{proof}

\begin{proof} [Proof of Lemma~\ref{lemme lambda}]
We write
$$\alpha=\displaystyle\max_{x \in S} \lambda_B(x)\, ,$$ 
and note that in fact $\lambda_B (x)\le \alpha$ for all $x\in \hat{S}$.
Now, let $t$ be as in the statement of the lemma.  One can distinguish the following two cases.

\medskip
$\bullet$
If $t \le 2s$, then $r=0$ and $q=t$, let $C$ consist of $t$ elements in $\hat{S}\setminus \{0\}$. Thus, we obtain
$$\displaystyle\sum_{c \in C} \lambda_B(c) \le \alpha t = \alpha\left(\frac{t(t+2s+6)+q(2s-q-2)}{4(s+1)}\right).$$

\medskip
$\bullet$
If $t \ge 2s+1$, then $r \ge 1$.  Since $S$ is super faithful the bounds $|j\hat{S}\setminus \{0\}|\ge \min(|G|-1,j(2s+2)-2) \ge j(2s+2)-2$ for $j=1,\dots , r$, and $|(r+1)\hat{S}\setminus \{0\}|\ge \min(|G|-1,(r+1)(2s+2)-2)\ge t$ hold.  Hence one may select a sequence of disjoint sets $C_j\ssq G\setminus \{0\}\, :j=1,\dots ,r+1$ such that $C_j\ssq j\hat{S}$ for each $j=1,\dots ,r+1$, and with $|C_j|=2s\, :j=1\, ,\,|C_j|=2s+2\, :j=2,\dots ,r\, ,\, |C_{r+1}|=q+2$.  Set $C=\bigcup_{j=1}^{r+1}C_j$, and note that $C\ssq G\setminus\{0\}$ has cardinality $t=r(2s+2)+q$.  Our proof of the lemma proceeds via proving upper and lower bounds on the quantity $\sum_{c\in C}\lambda_B (c)$.

The lower bound on $\sum_{c\in C}\lambda_B (c)$ is given immediately by Lemma~\ref{Proprietes de lambda}:
\[
\sum_{c\in C}\lambda_B (c)\ge |C||B|-|B|^2+|B|\, =tb-b^2+b\, .
\]

For the upper bound on $\sum_{c\in C}\lambda_B (c)$ we use the sub-additivity of $\lambda_B (x)$ (as in the proof of Lemma~\ref{lemme lambda general}) which gives us that $\lambda_B (c)\le j\alpha$ for all $c\in C_j$.  It follows that 
\begin{eqnarray*}
 \displaystyle\sum_{c \in C}\lambda_B(c) & \le &  2s \alpha + 2(2s+2) \alpha + \cdots + r(2s+2) \alpha + (q+2)(r+1)\alpha \\
                      &   =  & \frac{\alpha}{2}\left( r(r+1)(2s+2) + 2(q+2)(r+1)-4\right) \\
                      &   =  & \frac{\alpha}{2}\left( (r+1)(t+q+4) -4 \right)\\                      
                      &   =  & \alpha\left(\frac{t(t+2s+6)+q(2s-q-2)}{4(s+1)}\right)\, .
\end{eqnarray*}
Combining our bound on $\sum_{c \in C}\lambda_B(c)$ yields the inequality
\[
\alpha \ge \frac{4(s+1)b(t-b+1)}{t(t+2s+6)+q(2s-q-2)}\, ,
\]
as required.
\end{proof}

\end{document}